\documentclass[11pt,a4paper,twoside]{amsart}

\usepackage{graphicx}
\usepackage{amsfonts}
\usepackage{amssymb}
\usepackage{amsthm}
\usepackage[centertags]{amsmath}
\usepackage{newlfont}

\usepackage{color}
\usepackage{graphicx,color}
\usepackage{amsmath, amssymb, graphics}

 \def\r{\mathbb{R}}
 \def\l{\mathbb{L}}
 
\newtheorem{theorem}{Theorem}[section]
\newtheorem{proposition}[theorem]{Proposition}
\newtheorem{corollary}[theorem]{Corollary}

\theoremstyle{definition}
\newtheorem{definition}[theorem]{Definition}

\begin{document}
 
\title{ Riemann Zero Mean Curvature Examples in Lorentz-Minkowski Space}
 
\author{Seher Kaya}
\address{Department of Mathematics\\
Faculty of Science\\
Ankara University\\
06100 Ankara, Turkey}
\email{seherkaya@ankara.edu.tr}
\author{Rafael L\'opez}
 \address{Departamento de Geometr\'{\i}a y Topolog\'{\i}a\\
 Universidad de Granada\\
 18071 Granada, Spain}
\email{rcamino@ugr.es}
\thanks{Rafael L\'opez has  partially supported by  the grant no. MTM2017-89677-P, MINECO/AEI/FEDER, UE.}


\date{}

\begin{abstract} 
 Riemann zero mean curvature examples in the Lorentz-Minkowski space    are   surfaces with zero mean curvature    foliated by circles contained in parallel planes. In contrast to the Euclidean case, this family of surfaces presents new and rich features because of the variety of types of circles. In this paper,  we give  a geometric description of these examples  when the circles are contained in    spacelike planes  and  timelike planes.
 \end{abstract}

\maketitle

\section{Introduction and motivation}

In 1867 Riemann constructed a family of non-rotational minimal surfaces in Euclidean space foliated by circles contained in parallel planes (\cite{ri}).   In the literature, they are  known   as   Riemann minimal examples   and  play a remarkable role in the theory of minimal surfaces (\cite{hm,wm0}). 

When we extend   this type of surfaces in  Lorentz-Minkowski space $\l^3$,   we find two main differences   regarding to the Euclidean space. First,    the mean curvature is defined only in those surfaces  where    the induced metric from $\l^3$ is non-degenerated, that is, for spacelike surfaces (Riemannian metric) and for timelike surfaces (Lorentzian metric). Both types of surfaces have different behaviors.   For example, the Weingarten endomorphism is real diagonalizable for spacelike surfaces but it is not for timelike surfaces. On the other hand,  spacelike surfaces  with zero mean curvature   share similar properties with the minimal surfaces of Euclidean space, for example, they have    a variational characterization in terms of its area, being locally a maximum for the area. For this reason,   spacelike zero mean curvature surfaces  are called    {\it maximal surfaces}. By contrast,  it does not make sense to define the area of a timelike surface. 
 
A second difference comparing with the Euclidean context is    the notion of a circle. In Euclidean space, a circle is a planar curve whose points  are equidistant from a given point called the center of the circle. Such a definition can not extend to   $\l^3$ because of the existence of planes whose metric is degenerated. For this reason, it is more convenient to define   a    circle as a planar non-degenerate curve with nonzero constant curvature (\cite{lls,lo5}). Since  there are three types of planes according to    its  metric, there  are three types of circles in $\l^3$: see Section \ref{sec2} for details. Definitively, the family of  zero mean curvature surfaces in $\l^3$    foliated by circles contained in parallel planes is richer than the Euclidean case and this makes the interest to its study. 

To precise our terminology, we give the next definition.

\begin{definition} 
A Riemann zero mean curvature example (shortly a Riemann ZMC example) is a non-rotational surface in $\l^3$ with zero mean curvature and foliated by pieces of circles contained in parallel planes. 
\end{definition}

In this paper, we exclude the rotational surface which    are well known: see \cite{hn,ko,lo4}.  The first example   of a  spacelike Riemann ZMC example   was discovered by the second author in   \cite{lo1} where,  following ideas of Jagy (\cite{ja}), were   described all spacelike Riemann ZMC examples foliated by circles contained in spacelike planes.  Later,  and for any causal character of the circles,      spacelike Riemann ZMC  examples    were studied in \cite{lls} from the point of view of complex analysis obtaining the Weierstrass representation. More recently, Akamine has studied   the Riemann ZMC examples in terms of their causal characters (\cite{ak}) and he observed the existence of timelike Riemann ZMC examples foliated by circles with constant radii.

The Riemann ZMC examples also share a property with Riemann minimal examples  (\cite{en}). If a zero mean curvature surface in $\l^3$ is foliated by circles, then these circles are contained in parallel planes and, consequently, the surface is rotational   or it is a Riemann ZMC example (\cite{lo2,lo4,lls}).

 The aim of  the present paper is to give  a new approach of the Riemann ZMC examples when   the circles of the foliation are included in   spacelike planes or in timelike planes. In contrast to \cite{lls}, where the investigation was made only  for spacelike surfaces in terms of the Weierstrass representation, we see    the Riemann ZMC examples  as the zeroes of a regular function. Here we follow    similar ideas of Nitsche (\cite[pp.85-90]{ni}), and more recently, of Meeks and P\'erez (\cite{mp}).  In particular,  we obtain parametrizations of the Riemann ZMC examples without the use of complex notation and, consequently,  we will derive some of their geometric properties. Although the precise statements will appear in the subsequent sections, our main results are the following.
 \begin{enumerate}
 \item The parametrizations of the Riemann ZMC examples  are given in terms of elliptic integrals: see Theorems \ref{pr7}, \ref{pr77} and \ref{pr777}.
 \item We find the existence of   particular examples where   these integrals can be solved by quadratures, finding  explicit parametrizations of Riemann ZMC examples in terms of elementary functions: see Theorems \ref{pr3} and  \ref{pr4} and Propositions   \ref{pr9}, \ref{pr1111} and \ref{pr21}.  
 \item In contrast to   the Euclidean case,   we will find all   Riemann ZMC examples where the radii of the circles of the foliation are constant: see    Propositions \ref{pr1}, \ref{pr11} and \ref{pr111}.
 \end{enumerate}
 We organize this paper  as follows. In Section \ref{sec2}, we will obtain the expression of the   mean curvature of a surface in given as the zeroes of a smooth function. In Section \ref{sec3}, we study the Riemann ZMC examples foliated by circles contained in  spacelike planes. We will establish properties of the symmetries of these surfaces in Corollaries \ref{c1}, \ref{c2}, \ref{c3} and \ref{c4}. If   the circles are included in timelike planes, its study is separated in two sections, depending if the circles are spacelike (Section \ref{sec4}) or timelike (Section \ref{sec5}).

\section{Preliminaries}\label{sec2}

The Lorentz-Minkowski space $\l^3$ is   the vector space $\r^3$ with canonical coordinates  $(x_1,x_2,x_3)$ and endowed with the metric $\langle,\rangle=dx_1^2+dx_2^2-dx_3^2$. A vector $\vec{v}\in\r^3$ is spacelike, timelike or lightlike if $\langle \vec{v},\vec{v}\rangle$ is positive, negative or zero, respectively. The norm of $\vec{v}$   is $|\vec{v}|=\sqrt{\langle \vec{v},\vec{v}\rangle}$ if $\vec{v}$ is spacelike and $|\vec{v}|=\sqrt{-\langle \vec{v},\vec{v}\rangle}$ if $\vec{v}$ is timelike.  A curve or a surface   $A\subset\l^3$ is called  spacelike, timelike or lightlike if the induced metric on $A$ is Riemannian, Lorentzian or degenerated, respectively. This property of $A$ is called the causal character of $A$.    We refer  the reader to \cite{lo5} for some basics of $\l^3$. In $\r^3$, as affine space, we shall utilize the   terminology horizontal and vertical as usual, where the $x_3$-coordinate indicates the vertical direction and a horizontal direction is a direction parallel to the plane of equation $x_3=0$.

A {\it circle} in $\l^3$ is defined as a non-degenerate planar curve with nonzero constant curvature (\cite{lls,lo5}). After a rigid motion of $\l^3$, we  assume that the plane $P$ containing the circle is the plane   of equation  $x_3=0$, $x_1=0$ or   $x_2-x_3=0$, if $P$ is spacelike, timelike or lightlike, respectively. Consequently, a circle $C\subset\l^3$ can be described as follows:
\begin{enumerate}
\item If $P\equiv \{x_3=0\}$, then $C$ is an Euclidean circle $\alpha(s)=p+r( \cos(s),\sin(s),0)$, with center $p\in P$ and radius $r>0$.
\item If $P\equiv \{x_1=0\}$, then $C$ is a  hyperbola $\alpha(s)=p+r(0,\sinh(s),\cosh(s))$ if $\alpha$ is spacelike or $\alpha(s)=p+r(0,\cosh(s),\sinh(s))$ if $\alpha$ is timelike. Here $p\in P$ is the center and $r>0$ is the radius.

\item If $P\equiv \{x_2-x_3=0\}$, then $C$ is a  parabola $\alpha(s)=p+(s,r s^2,rs^2)$, $p\in P$ and $r>0$.  
\end{enumerate}

We say that a surface is foliated by (pieces of) circles if it is constructed by a smooth one-parameter family of (pieces of) circles. In the case of the Riemann ZMC examples, the planes containing the circles are parallel. 

Let $M$ be an  orientable surface in $\l^3$  whose induced metric $\langle,\rangle$ is  non-degenerated.  If $X=X(u,v)$ is a local parametrization of $M$, let $g_{11}=\langle X_u,X_u\rangle$, $g_{12}=\langle X_u,X_v\rangle$ and $g_{22}=\langle X_v,X_v\rangle$ be  the coefficients of the first fundamental form  with respect to   $X$ and $W=g_{11}g_{22}-g_{12}^2$. Then $M$ is spacelike (resp. timelike) if $W>0$ (resp. $W<0$). In both types of surfaces, the mean curvature $H$   is defined as the trace of the second fundamental form. If $H=0$ everywhere, we say that $M$ has zero mean curvature (ZMC in short).

We now consider a surface given as an implicit equation and we calculate the expression of its mean curvature $H$.   Let $F:O\subset\r^3\rightarrow\r$ be a smooth function defined in an open set $O$  of $\r^3$ and let $M=F^{-1}(\{0\})$ be a surface defined as the preimage of a regular value of $F$. Suppose that $M$ endowed with the induced metric of $\l^3$ is a non-degenerate surface. Then the Lorentzian gradient   $\nabla^LF$ of $F$ 
$$\nabla^{L}F=\left( F_{x_1}, F_{x_2},- F_{x_3}\right)$$
defines a normal vector field on $M$ where  the subscript $x_i$ indicates the partial derivative with respect to the $x_i$-variable.  Then $\nabla^LF/|\nabla^LF|$ is a unit normal vector field on $M$ and   
$$ \mbox{div}^L\left(\frac{\nabla^{L}F}{|\nabla^{L}F|}\right)=\bigg(\frac{F_{x_1}}{|\nabla^{L}F|}\bigg)_{x_{1}}+\left(\frac{F_{x_2}}{|\nabla^{L}F|}\right)_{x_{2}}-\left(\frac{F_{x_3}}{|\nabla^{L}F|}\right)_{x_{3}}=H,$$
 where   $\mbox{div}^L$ is the Lorentzian     divergence operator. As a consequence, the equation $H=0$ is equivalent to
$$ \frac{\Delta^{L}F}{|\nabla^{L}F|}+\frac{\epsilon}{|\nabla^{L}F|^{3}} (\nabla^LF)^t\cdot \mbox{Hess} F\cdot \nabla^LF=0,$$
where    $\epsilon=1$ if $M$ is  spacelike   and $\epsilon=-1$ if $M$ is  timelike,  $\Delta^LF=F_{x_1x_1}+F_{x_2x_2}-F_{x_3x_3}$ and
$$\mbox{Hess}F=
\left(
\begin{array}{ccc}
 F_{x_1x_1}&F_{x_1x_2}&F_{x_1x_3}  \\
 F_{x_2x_1}&F_{x_2x_2}&F_{x_2x_3}\\
 F_{x_3x_1}&F_{x_3x_2}&F_{x_3x_3}
\end{array}
\right).$$

\begin{proposition}
If   $M= F^{-1}(\{0\})$ is a non-degenerate surface in  $\l^3$, then  $M$ is a ZMC surface    if and only if
\begin{equation}\label{elm1}
-\langle \nabla^{L}F,\nabla^LF\rangle\Delta^{L}F+ (\nabla^LF)^t\cdot \mbox{Hess} F\cdot \nabla^LF=0.
\end{equation}
\end{proposition}

\section{ Riemann ZMC examples   foliated by   circles contained in spacelike planes }\label{sec3}

In this section, we study  ZMC surfaces in $\l^3$   foliated by circles   contained in parallel  spacelike planes.  After a rigid motion of   $\l^3$, we  suppose that the foliating circles are contained in parallel planes   to the plane of equation $x_3=0$, hence the circles are Euclidean circles.  Let $M$ be a such surface and  consider the   $x_3$-coordinate   as a parameter of the foliation. Let $z=x_3$ and  write $(\alpha(z),z)=(\alpha_1(z),\alpha_2(z),z)$ the center of the circle $M\cap\{x_3=z\}$ and by $r(z)>0$ its radius. Here  $\alpha$ and $r$ are smooth functions defined in an interval $(a,b)\subset\r$. If  $F:\r^2\times(a,b)\rightarrow\r$ is  the function   
$$F(x,z)= (x_1-\alpha_1(z))^2+(x_2-\alpha_2(z))^2-r(z)^2,$$
where $x=(x_1,x_2)$, then $M\subset F^{-1}(\{0\})$. We observe that if $\vec{a}=(0,0)$, then $\alpha$ is the $x_3$-axis, which corresponds with the case that the surface is rotational.

Let $\alpha'=d\alpha/dz$. Here  we identify the factor $\r^2$ of the domain of $F$ with  $\r^2\times\{0\}$ endowed with the induced metric    $\langle,\rangle=dx_1^2+dx_2^2$. The computation of   each one  of the terms of (\ref{elm1}) yields 
\begin{eqnarray*}
&&\nabla^{L}F=(2(x-\alpha(z)),-F_{z})\\
&& \Delta^{L}F=4-F_{zz}\\
&&(\nabla^LF)^t\cdot \mbox{Hess} F\cdot \nabla^LF = 8r^{2}+8F_{z}\langle x-\alpha,\alpha'\rangle+F_{z}^{2}F_{zz}.
\end{eqnarray*}
   Then  equation (\ref{elm1}) writes as
\begin{equation}\label{elm2}
F_{z}^{2}+r^{2}(F_{zz}-2)+2F_{z}\langle x-\alpha,\alpha'\rangle=0.
\end{equation}
By using  the definition of $F$, 
\begin{equation}\label{fz}
F_{z}=-2\langle x-\alpha,\alpha'\rangle-(r^2)',
\end{equation}
and   (\ref{elm2}) simplifies into  $-2r^{2}+r^{2}F_{zz}-(r^{2})'F_{z}=0$. We divide    by $r^{4}$, obtaining easily
$$
-\frac{2}{r^{2}}+\bigg(\frac{F_{z}}{r^{2}}\bigg)_{z}=0.
$$
Hence we   integrate  with respect to $z$,  
\begin{equation}\label{elm3}
-2\int^z  \frac{1}{r(u)^{2}}du+\frac{F_{z}}{r^{2}}=c(x)
\end{equation}
for some  function $c=c(x)$ depending   on $x$.  The definition of $F$ leads to
$$\left(\frac{F_{z}}{r^{2}}\right)_{x_i}=-\frac{2\alpha_{i}'}{r^{2}},\quad i=1,2,$$
and thus
$$
\frac{F_{z}}{r^{2}}=-\frac{2}{r^2}\langle x,\alpha'\rangle+ h(z)
$$
for some  function $h=h(z)$. If we insert  this in (\ref{elm3}), it follows that
$$-2\int^z  \frac{1}{r(u)^{2}}du-\frac{2}{r^2}\langle x,\alpha'\rangle+ h(z)=c(x).$$
Differentiating with respect to the variables $x_1$ and $x_2$,
$$c_{x_{i}}=-\frac{2\alpha_{i}'(z)}{r(z)^2},\quad i=1,2.$$
Therefore there is a vector $\vec{a}=(a_1,a_2)$   such that
$$\left(c_{x_{1}},c_{x_{2}}\right)=-\frac{2}{r(z)^2}\alpha'(z)=\vec{a},$$
that is,
$$\alpha'(z)=-r(z)^2 \vec{a}.$$
 Hence integrating with respect to $z$, 
$$\alpha(z)=-m(z)\vec{a}, \quad m(z)=\int^z r(u)^2du.$$
  As immediate consequence of the expression of $\alpha$, the curve $\alpha$ is a horizontal straight-line and thus the curve of the centers of the circles is contained in the plane containing the $x_3$-axis.  

\begin{proposition} \label{pr33}
The curve formed by the centers of the foliation  circles of a  Riemann ZMC example foliated by   circles contained in spacelike planes is planar.
\end{proposition}

We  express the radius of the circle $M\cap\{x_{3}=z\}$ as a function of $z$. We will identify $\vec{a}$ with the vector $(\vec{a},0)\in \l^3$ in the $x_1x_2$-plane.  Moreover, $\langle\vec{a},\vec{a}\rangle=|\vec{a}|^2$. It follows from (\ref{fz}) that 
\begin{eqnarray*}&&F_z=2\langle x+m\vec{a},r^2\vec{a}\rangle-(r^2)'\\
&&F_{zz}=2r^4\langle \vec{a},\vec{a}\rangle+2(r^2)'\langle x+m\vec{a},\vec{a}\rangle-(r^2)''.
\end{eqnarray*}
Inserting $F_z$ and $F_{zz}$ in  (\ref{elm2}), 
$$2 |\vec{a}|^2  r^{6}+(r^{2})'^{2}-r^{2}(2+(r^{2})'')=0.$$
By taking $q=r^2$, we deduce 
\begin{equation}\label{qq}
2\langle \vec{a},\vec{a}\rangle q^3+q'^2-q(2+q'')=0.
\end{equation}
  Definitively, equation (\ref{qq}) describe all  ZMC surfaces foliated   by circles contained in parallel    spacelike planes.

As immediate consequence of   (\ref{qq})  is the existence of solutions where   the radius function $r$ is constant,   which is a novelty comparing with the Euclidean case.

\begin{proposition}[Case of constant radii]\label{pr1}
The only Riemann ZMC examples   foliated by  circles with constant radii contained in   spacelike planes    are parametrized by  
\begin{equation}\label{tc}
X(z,v)= z(-r^2\vec{a},1)+r(\cos(v),\sin(v),0),
\end{equation}
where $z,v\in\r$ and $\vec{a}\in\r^2\setminus\{0\}$. The surfaces are timelike  and can be extended to two lightlike straight-lines. See figure \ref{fig1}, left.  
\end{proposition}

\begin{proof} If $r$ is constant, then  $|\vec{a}|^2  q^2-1=0$. Thus $\vec{a}\not=0$ and $r^2=1/|\vec{a}|$. Then $m(z)=r^2z$ and $\alpha(z)=-r^2z\vec{a}$ obtaining   (\ref{tc}). 
 In order to know the causal character of the surface, we have $X_z=(-r^2\vec{a},1)$ and $X_v=r(-\sin(v),\cos(v),0)$. Hence, $g_{11}=0$, $g_{22}=r^2$ and $g_{12}=-r^3\langle\vec{a},(-\sin(v),\cos(v))\rangle$. We infer that $M$ is timelike except at the points where $g_{12}=0$, which are lightlike points.

Therefore the surface can be extended to the points $\langle\vec{a},(-\sin(v),\cos(v))\rangle=0$ as  a region of lightlike points.  These points   form a set of straight-lines, such as was demonstrated in  \cite{ak}. Indeed,   there are exactly two values  $v_0$ and  $v_1$, up to an integer $2\pi$-multiple, such that  $\langle\vec{a},(-\sin(v_i),\cos(v_i))\rangle=0$, $i=1,2$. This  set    parametrizes as $z\mapsto X(z,v_i)=r(\cos(v_i),\sin(v_i),0)+z(-r^2\vec{a},1)$, proving that  they are two lightlike straight-lines.  
\end{proof}

\begin{figure}[hbtp]
\includegraphics[width=.4\textwidth]{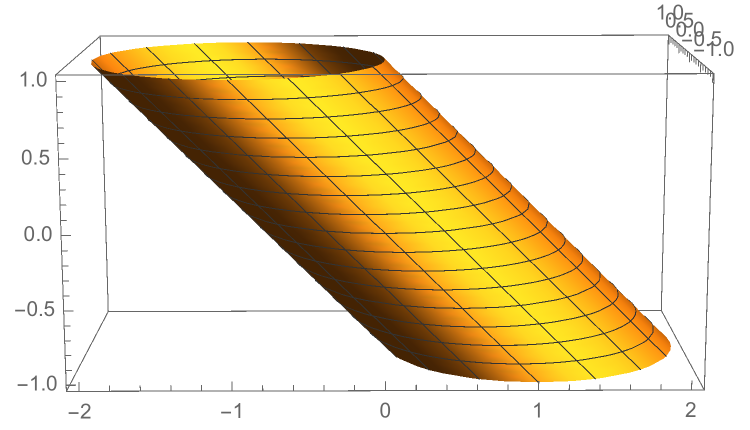}\ \includegraphics[width=.2\textwidth]{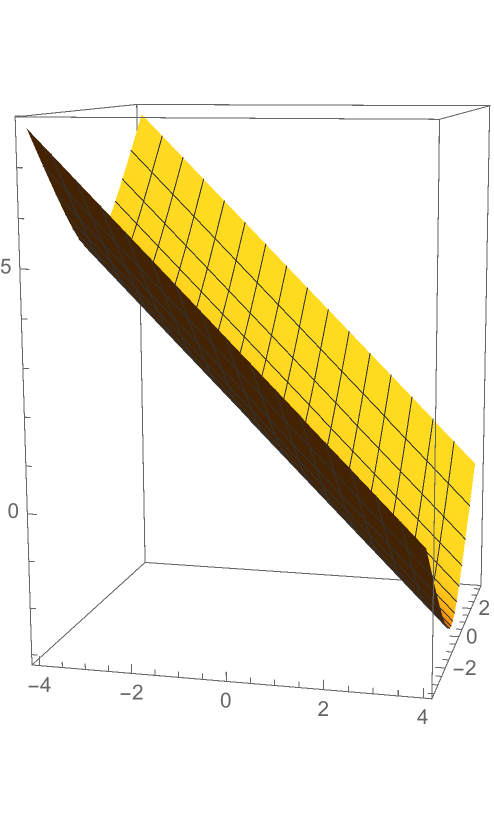}\ \includegraphics[width=.35\textwidth]{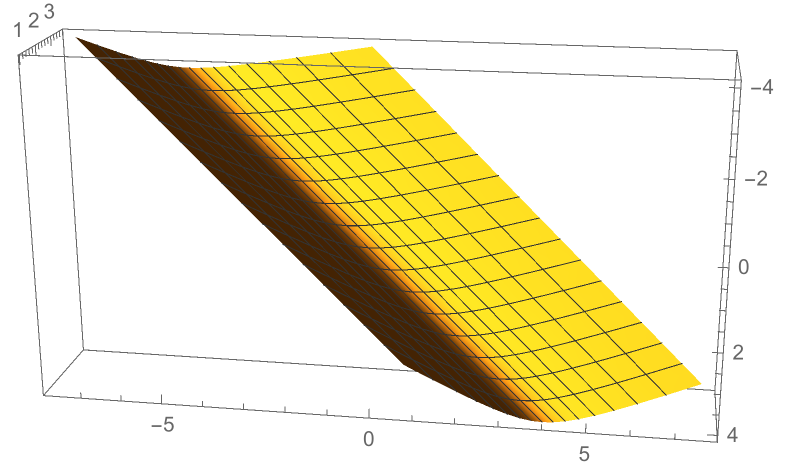}
\caption{Riemann ZMC examples with constant radii. Left: Euclidean circles. Middle:  spacelike  hyperbola. Right: timelike hyperbola}\label{fig1}
\end{figure}

From now, we suppose that the radii of the foliation are not constant.  We write  (\ref{qq}) as
$$\bigg(\frac{(q')^{2}}{q^{2}}\bigg)'=\frac{2q'}{q^{3}}(q''q-q'^{2})=4q'\left( |\vec{a}|^2 -\frac{1}{q^{2}}\right).$$
Integrating with respect to $z$, 
\begin{equation}\label{difc}
\frac{(q')^{2}}{q^{2}}=4\bigg( |\vec{a}|^2  q+\frac{1}{q}\bigg)+4\lambda
\end{equation}
for a constant $\lambda\in\r$. In particular, the right-hand side of \eqref{difc} must be non-negative. Now we have
$$q'=\frac{dq}{dz}=2\sqrt{ |\vec{a}|^2  q^{3}+\lambda q^{2}+q},$$
or equivalently, 
$$\frac{dz}{dq}=\frac{1}{2}\frac{1}{\sqrt{ |\vec{a}|^2  q^{3}+\lambda q^{2}+q}}.$$
As the new parameter is $q$,   the center of the circle $M\cap\{x_3=z\}$ is
$(-m(q)\vec{a},z(q))$, where
\begin{equation}\label{mz}
z(q)=\frac{1}{2}\int ^q\frac{du}{\sqrt{ |\vec{a}|^2  u^{3}+\lambda u^{2}+u}},\quad m(q)=\frac12\int^q\frac{u\ du}{\sqrt{|\vec{a}|^2  u^3+\lambda u^2+u}}.
\end{equation}
The parametrization   of $M$ is  
$$X(q,v)=(-m(q)\vec{a},z(q))+\sqrt{q}(\cos(v),\sin(v),0).$$

After a rotation about the $x_3$-axis,   suppose that $\vec{a}=a(1,0)$, $a>0$. With a change of variables, we see that   $X_{a,\lambda}$ and $X_{1,\lambda/a}$ is  $\sqrt{a} X_{a,\lambda}(q,v)=X_{1,\lambda/a}(aq,v)$. Hence  the corresponding surfaces are equal up to a homothety of ratio $\sqrt{a}$. Fixing the value of $a$, say $a=1$, we conclude from \ref{mz}  that the family of Riemann ZMC examples  only depends on the real parameter $\lambda$. Now 
\begin{equation}\label{mz2}
 z(q)=\frac{1}{2}\int ^q_{q_0}\frac{du}{\sqrt{u^{3}+\lambda u^{2}+u}},\quad m(q)=\frac12\int^q_{q_0}\frac{u\ du}{\sqrt{u^3+\lambda u^2+u}}.
\end{equation}
Suppose now that the surface is spacelike. The first derivatives of $X$ are 
$$X_q=(-m',0,z')+\frac{1}{2\sqrt{q}}(\cos(v),\sin(v),0)$$
$$ X_v=\sqrt{q}(-\sin(v),\cos(v),0).$$
From     (\ref{mz2}), the spacelike condition   $g_{11}g_{22}-g_{12}^2>0$ is 
\begin{equation}\label{spa}
(1+\cos(v)^2)q-2\cos(v)\sqrt{q^2+\lambda q+1}+\lambda>0.
\end{equation}
The radicand of this two-degree inequation  is $q^3+\lambda q^2+q$. We analyze if  it   is positive in order to determined the lower limit $q_0$ of the integrals (\ref{mz2}). The equation $u^3+\lambda u^2+u=0$ has three roots, namely, $0$ and, if exist,
$$q_1=\frac{-\lambda-\sqrt{\lambda^2-4}}{2},\quad q_2=\frac{-\lambda+\sqrt{\lambda^2-4}}{2},$$
with $q_1\leq q_2$. 

We now study the properties of the surface depending   if $\lambda^2-4<0$, $\lambda^2-4=0$ and $\lambda^2-4>0$. Firstly, we distinguish the case $\lambda^2-4=0$ because we will obtain explicit parametrizations of surfaces. This contrast to the Euclidean case, where the Riemann minimal examples are given in terms of elliptic integrals that can not be integrated by simple quadratures.  

\begin{theorem}\label{pr3} There are  spacelike  Riemann ZMC examples foliated by circles contained  in   spacelike planes with   parametrizations  given in terms of elementary functions for the cases $\lambda=2$ and $\lambda=-2$.  
\begin{enumerate} 
\item Case $\lambda=2$.  The surface is 
\begin{equation}\label{ex1}
X(r,v)=(-r+\arctan(r),0,\arctan(r))+r(\cos(v),\sin(v),0),
\end{equation}
where $r>0$, $v\in (0,2\pi)$. See figure \ref{fig3}, left. The properties of this surface are the following.
\begin{enumerate}
\item The circles of the foliation are punctured. 
\item The surface is contained in the horizontal slab $0<x_3\leq \pi/2$.
\item The surface      converges to a  conelike point as $x_3\rightarrow 0$ and converges to a    straight-line  $L$ orthogonal to the plane $\Pi$ of equation $x_2=0$ as $x_3\rightarrow \pi/2$. The straight-line $L$ is contained in the surface.
\item The surface is    asymptotic to the horizontal plane of equation $x_3=\pi/2$.  
\item The  surface can be extended  to  a lightlike straight-line.  
\end{enumerate}
\item Case $\lambda=-2$.  The surface is 
\begin{equation}
X(r,v)=\left(-r-\frac12\log\left(\frac{r-1}{r+1}\right),0,\frac12\log\left(\frac{r-1}{r+1}\right)\right)+r(\cos(v),\sin(v),0), \label{ex2}
\end{equation}
where $r\in [r_0,\infty)$ for any $r_0>1$ and $\cos(v)<1-2/r^2$. See figure \ref{fig3}, right. The properties of this surface are the following.
\begin{enumerate}
\item The surface is foliated by pieces of circles that can be extended to fully circles assuming that the surface is degenerated  in the points $\cos(v)=1-2/r^2$ and timelike when $\cos(v)>1-2/r^2$. 
\item The surface lies contained in the horizontal slab $ z_0<x_3<0$, where $z_0=\log(\frac{r_0-1}{r_0+1})/2$. 
\item If $r=r_0$, the surface has a boundary component that is  (part of) a circle and  it  converges to a straight-line orthogonal to the plane $\Pi$ as $x_3\rightarrow 0$. This straight-line is  contained in the surface. 
\end{enumerate}
\end{enumerate}

\end{theorem}
\begin{proof}

\begin{enumerate}
\item Case  $\lambda=2$. Then $q_1=q_2=-1$ and we take $q_0=0$ as the lower limit of integration  in (\ref{mz2}).  An integration by quadratures  gives 
\begin{eqnarray*}
&&z(q)=\frac{1}{2}\int ^q_0\frac{du}{\sqrt{u}(u+1)}=\arctan(\sqrt{q})\\
&&m(q)=\frac12\int^q_0\frac{\sqrt{u}\ du}{u+1}=\sqrt{q}-\arctan(\sqrt{q}).
\end{eqnarray*}
Setting  $r=\sqrt{q}$,   the parametrization of the surface is (\ref{ex1}). Inequality (\ref{spa}) writes as
$$q(1-\cos(v))^2+2(1-\cos(v))>0,$$
which holds if   $\cos(v)\not=1$.  This proves that the circles of the foliation are punctured.

On the other hand,   the excluded points $\cos(v)=1$ form a curve  where the metric is degenerated. Moreover, as the (extended) surface has only spacelike and lightlike points, the lightlike points correspond with a straight-line (\cite{ak}). Indeed, these points   parametrize as $r\mapsto X(r,v)=\arctan(r)(1,0,1)$ and    from (\ref{ex1}), we have  
$$\lim_{r\rightarrow 0}x_3(r)= 0,\quad \lim_{r\rightarrow\infty}x_3(r)= \frac{\pi}{2}.$$
 This proves that the surface is included in the slab $ 0<x_3<\pi/2$. Finally, we consider the intersection of the plane $\Pi$   with each circle of the foliation, that is, for the points where $\cos(v)=\pm 1$. If $r\rightarrow\infty$, the points where $\cos(v)=-1$ go to $-\infty$ and the points where $\cos(v)=1$ (lightlike points)  converge to the point $(\pi/2,0,\pi/2)$. If $r\rightarrow\infty$,     the surface converges  to the straight-line $L$  orthogonal to the plane $\Pi$ through the point $(\pi/2,0,\pi/2)$.

\item Case $\lambda=-2$. Then $q_1=q_2=1$. There are  two cases, namely,    $q_0=0$ and $q\in (0,1)$, or $q_0=1$ and $q\in (1,\infty)$. We observe that if $q\in (0,1)$, then (\ref{spa}) is equivalent to
$$(1+\cos(v))\left(q(1+\cos(v))-2\right)>0.$$
It is clear that  $\cos(v)\not=-1$. Then  the above inequality writes as
$q>2/(1+\cos(v))$, which is impossible since $q\in (0,1)$. This proves   definitively that if the surface is spacelike, then $q\in (1,\infty)$.  The spacelike condition (\ref{spa}) is now
$$(1-\cos(v))\left((1-\cos(v))q-2\right)>0,$$
in particular, $\cos(v)\not=1$, hence $\cos(v)<(q-2)/q$.  Consequently, the spacelike condition implies that in each leaf of the foliation the corresponding circle is not complete, and the spacelike character is only defined in arcs of this circle. Moreover,     the lengths of these (spacelike) arcs go varying along the foliation.   Assume  that $\cos(v)<(q-2)/q$, hence 
$v\in\left(\arccos((q-2)/2),2\pi \arccos((q-2)/2)\right)$. It follows
$$z(q)= \frac12\log\left(\frac{\sqrt{q}-1}{\sqrt{q}+1}\right),\quad m(q)=\  \sqrt{q}+\frac12\log\left(\frac{\sqrt{q}-1}{\sqrt{q}+1}\right).$$
It is obvious that $q_0>1$   since  if $q_0=1$, the integrals are indefinite.    If  $r=\sqrt{q}$, the parametrization of the surface coincides with (\ref{ex2}).  Moreover,  $z_0\leq x_3<0$, where $z_0= \log((r_0-1)/(r_0+1))/2$ and $r_0=\sqrt{q_0}$.

If $r\rightarrow\infty$, the points satisfying $\cos(v)=-1$ go to $-\infty$ but   the points with $\cos(v)=1$ (lightlike points)  converge to the point $(0,0,0)$. Since   $r\rightarrow\infty$,     the surface has a limit set the straight-line $L$ orthogonal to the plane $\Pi$ through the point $(0,0,0)$. 
 \end{enumerate}  
  \end{proof}  

\begin{figure}
\begin{center}
\includegraphics[width=.4\textwidth]{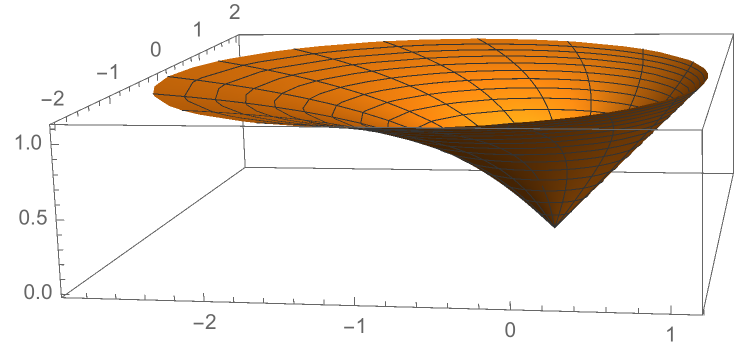},\quad \includegraphics[width=.55\textwidth]{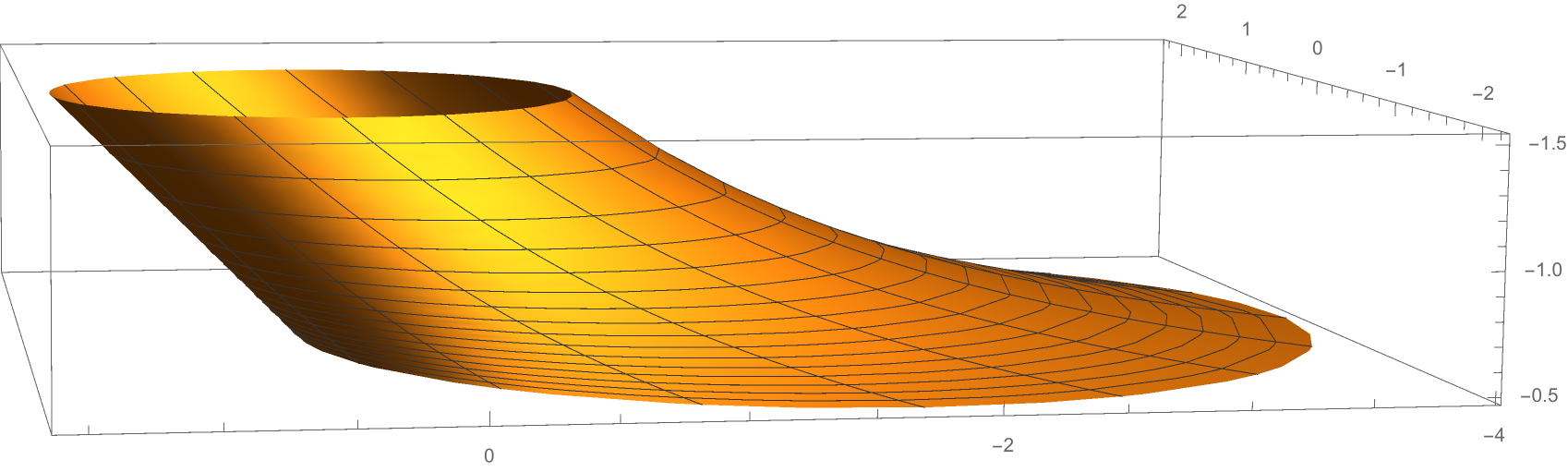}
\end{center}
\caption{Riemann ZMC examples foliated by Euclidean circles of Theorem \ref{pr3}. Both surfaces extend to lightlike straight-lines}\label{fig3}
\end{figure}

Until here, we have assumed that the surface is spacelike. However, the same arguments hold by changing  the domain of the parametrization in order to ensure $W<0$. Recall that if $\lambda=2$, then the surface can not extend to timelike points. However, this differs if    $\lambda=-2$. In such a case, if $q\in (1,\infty)$, there are regions of timelike points and if  $q\in (0,1)$, the surface is timelike except when   $\cos(v)=-1$.   Here   we take $q_0=0$ to be the lower limit of integration  in (\ref{mz2}),  obtaining
$$z(q)= \operatorname{artanh} (\sqrt{q}), \quad  m(q)=  -\sqrt{q}+\operatorname{artanh} (\sqrt{q}).$$

\begin{theorem}\label{pr4} In case  $\lambda=-2$, we have     parametrizations  of timelike Riemann ZMC examples foliated by   circles contained  in   spacelike planes in terms of elementary functions:
\begin{enumerate}
\item       The parametrization  (\ref{ex2}) for any $r_0>1$, when  $r\in [r_0,\infty)$ and  $1-2/r^2<\cos(v)<1$.
\item The parametrization   
$$X(r,v)=(r-\operatorname{artanh} (r),0,\operatorname{artanh}(r))+r(\cos(v),\sin(v),0),$$
 where $r\in (0,1)$ and $v\in(-\pi,\pi)$. This surface extends to a lightlike straight-line by considering the points $\cos(v)=-1$.  The surface is included in the halfspace $x_3>0$ with $\lim_{r\rightarrow\infty}x_3=\infty$. See figure \ref{fig5}  and  \cite[Ex. 1]{lo3}.   
\end{enumerate}
\end{theorem}

\begin{figure}[hbtp]
\begin{center}\includegraphics[width=.4\textwidth]{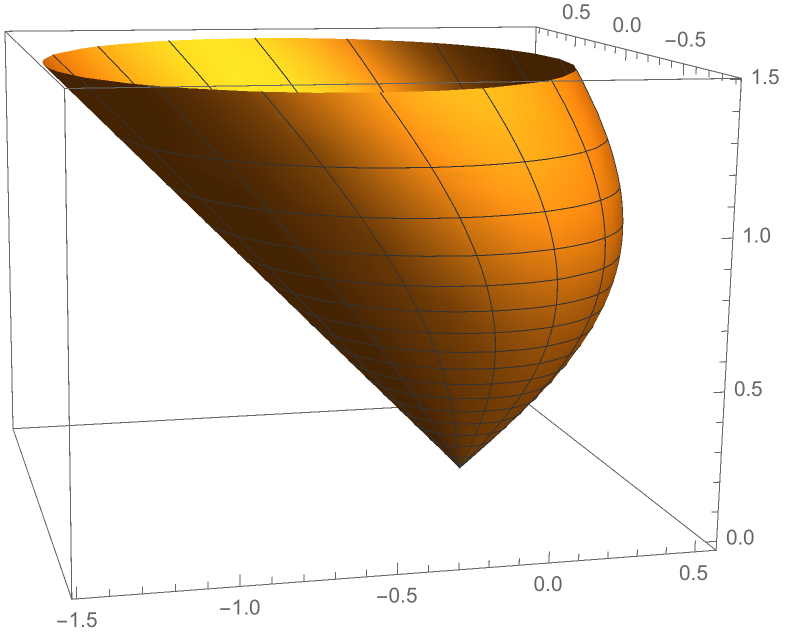}
\end{center}
\caption{Theorem \ref{pr4}: a timelike Riemann ZMC example foliated by   circles contained in   spacelike planes and with explicit parametrization}\label{fig5}
\end{figure}

We  consider the general case for the parameter $\lambda$.  In the following result, we show the geometric properties of these surfaces.

\begin{theorem}\label{pr7}
  Riemann ZMC examples foliated by circles contained in   spacelike planes form a one-parameter family of surfaces depending on a parameter $\lambda$ and parametrize as 
  $$X (q,v)=(m(q),0,z(q))+\sqrt{q}(\cos(v),\sin(v),0),$$
where
$$ z(q)=\frac{1}{2}\int ^q\frac{du}{\sqrt{u^{3}+\lambda u^{2}+u}},\quad m(q)=\frac12\int^q\frac{u\ du}{\sqrt{u^3+\lambda u^2+u}}.$$
 Depending on $\lambda$, we have the following cases.  
  
\begin{enumerate}
\item Case $\lambda^2=4$. These surfaces have been described in Theorems \ref{pr3} and \ref{pr4}.
\item Case $\lambda^2<4$. The surface contains regions of points with the  three causal characters; it is included in a horizontal slab $\  z_0\leq x_3<0$ with a conelike point at $x_3=0$; the surface contains  a straight-line orthogonal to the plane $\Pi$. See figure \ref{fig6}, left.
\item Case $\lambda<-2$. The $q$-parameter belongs to $(0,q_1)\cup (q_2,\infty)$. 
\begin{enumerate}
\item In the interval $(0,q_1)$ the surface is timelike and included in a horizontal slab $0<x_3<z_0$. As $x_3\rightarrow 0$, the surface converges to a point and if $x_3\rightarrow c$, the surface converges to   a circle.  
\item If $q\in (q_2,\infty)$, the surface has regions of points with the  three causal characters. Moreover, it is included in a horizontal slab  $\ 0\leq x_3\leq z_0$. If  $x_3=0$,  the   surface   has a lower boundary component that is   a circle and as $x_3= z_0$, the surface is a straight-line orthogonal to the plane $\Pi$ at the height $z_0$.   See figure \ref{fig6}, right.
\end{enumerate}
\item Case $\lambda>2$. The surface is spacelike and included in a horizontal slab  $ 0<x_3\leq z_0$. As  $x_3\rightarrow 0$, the surface converges to a conelike point  and if $x_3= z_0$, the surface is  a straight-line $L$ orthogonal  to the plane $\Pi$. 
\end{enumerate}
\end{theorem}

\begin{proof}   Let $M$ be a Riemann ZMC example foliated by circles contained in parallel planes   to the plane of equation $x_3=0$. 
The centers of the circles $M \cap\{x_3=z(q)\}$ lie included in the plane $\Pi$ of equation $x_2=0$. Assume $\lambda^2\not=4$.

\begin{enumerate}
\item Case $\lambda^2-4<0$. Since $u^2+\lambda u+1$ has not real roots, then $u^2+\lambda u+1>0$. This implies  that $u^3+\lambda u^2+u$ only vanishes at $u=0$, hence  the radicand $u^3+\lambda u^2+u$ is positive when $u>0$. Then we can   choose $q_0=0$ to be the lower limit in the integrals in (\ref{mz2}). In view of 
 $$\lim_{q\rightarrow\infty}\frac{1}{2}\int_0^q\frac{du}{\sqrt{u^3+\lambda u^2+u}}:=z_0<\infty,$$ 
the surface $M$ lies contained in a slab of the form $0<z<z_0$    and $M$ is  asymptotic to the horizontal plane of equation $x_3=z_0$.  

Each circle $M\cap\{x_3=z(q)\}$ meets $\Pi$ in two antipodal points,   $A_{+}(q)$ and $A_{-}(q)$, by taking $\cos(v)=1$ and $\cos(v)=-1$ respectively:
$$A_{\pm}(q)=\left(\pm\sqrt{q}+\frac12\int_0^q\frac{u}{\sqrt{u^3+\lambda u^2+u}}du,0,\frac12\int_0^q\frac{1}{\sqrt{u^3+\lambda u^2+u}}du\right).$$
Then 
\begin{eqnarray*}
&&\lim_{q\rightarrow\infty}\sqrt{q}+\frac12\int_0^q\frac{u}{\sqrt{u^3+\lambda u^2+u}}du=\infty\\
&&\lim_{q\rightarrow\infty}-\sqrt{q}+\frac12\int_0^q\frac{u}{\sqrt{u^3+\lambda u^2+u}}du:=c<\infty
\end{eqnarray*}
for some $c<0$. Thus 
$$\lim_{q\rightarrow\infty} A_{-}(q):=A=(c,0,z_0),\quad  \lim_{q\rightarrow\infty}A_{+}(q)=\infty.$$ 
This implies that $M\cap\{x_3=z_0\}\not=\emptyset$. Since $A_{+}(q)$ diverges, then $M\cap\{x_3=z_0\}$ is  a straight-line $L$ orthogonal to the plane $\Pi$ through the point $A$.

\item Case  $\lambda<-2$. Then $0<q_1<q_2$ and the polynomial $u^3+\lambda u^2+u$ is positive in $(0,q_1)\cup (q_2,\infty)$. 
\begin{enumerate}
\item Case   $q\in (0,q_1)$.   The right-hand side of (\ref{spa}) is negative, hence   the surface is timelike. As the integrals in (\ref{mz2}) are finite,   the surface lies contained in the   slab   $0<x_3<z_0$ where $z_0=x_3(q_1)$. 
\item Case    $q\in (q_2,\infty)$. We take $ q_2$ to be the lower limit of the integrals in (\ref{mz2}). This implies that the initial circle  of the foliation has radius $\sqrt{q_2}$.   Now  
$$ \lim_{q\rightarrow\infty} \frac12\int_{q_2}^q\frac{1}{\sqrt{u^3+\lambda u^2+u}}du=: z_0<\infty,$$
proving that $M$ is included in the  horizontal slab $0<x_3<z_0$. Moreover, 
\begin{eqnarray*}
&&\lim_{q\rightarrow\infty}\sqrt{q}+\frac12\int_{q_2}^q\frac{u}{\sqrt{u^3+\lambda u^2+u}}du=\infty\\
&&\lim_{q\rightarrow\infty}-\sqrt{q}+\frac12\int_{q_2}^q\frac{u}{\sqrt{u^3+\lambda u^2+u}}du:=c
\end{eqnarray*}
for some $c\in\r$. A similar  argument  as in the case $\lambda^2-4<0$, proves that the surface is asymptotic to a straight-line orthogonal to the plane $\Pi$ at the point $(c,0,z_0)$. 

\end{enumerate}
\item Case $\lambda>2$. The roots of the radicand $u^3+\lambda u^2+u$ are $0$,   $q_1$ and $q_2$ with $q_1<q_2<0$. Since $q$ is positive,   we may choose $q_0=0$ to be the lower limit in the integrals (\ref{mz2}).  Now the spacelike condition (\ref{spa}) holds for any $q>0$, indeed, (\ref{spa}) is equivalent to $((1+\cos(v)^2)q+\lambda>2\cos(v)\sqrt{q^2+\lambda q+1}$. This inequality holds trivially if $\cos(v)\leq 0$. If $\cos(v)>0$, squaring and simplifying, we obtain $\sin(v)^4q^2+2\lambda q\sin(v)^2+\lambda^2-4\cos(v)^2$, which is positive if $\lambda>2$. Again, we obtain 
\begin{eqnarray*}
&&\lim_{q\rightarrow\infty} \frac12\int_{0}^q\frac{1}{\sqrt{u^3+\lambda u^2+u}}du:=z_0<\infty\\
&&\lim_{q\rightarrow\infty}\sqrt{q}+\frac12\int_{0}^q\frac{u}{\sqrt{u^3+\lambda u^2+u}}du=\infty\\
&&\lim_{q\rightarrow\infty}-\sqrt{q}+\frac12\int_{0}^q\frac{u}{\sqrt{u^3+\lambda u^2+u}}du:=c
\end{eqnarray*}
for some $z_0,c\in\r$. This proves   that $M$ is  included in the slab  $0<x_3<z_0$ and $M$ is asymptotic to a straight-line orthogonal to the plane $\Pi$ at the height $x_3=z_0$.

\end{enumerate}
  \end{proof}

\begin{figure}[hbtp]
\includegraphics[width=.26\textwidth]{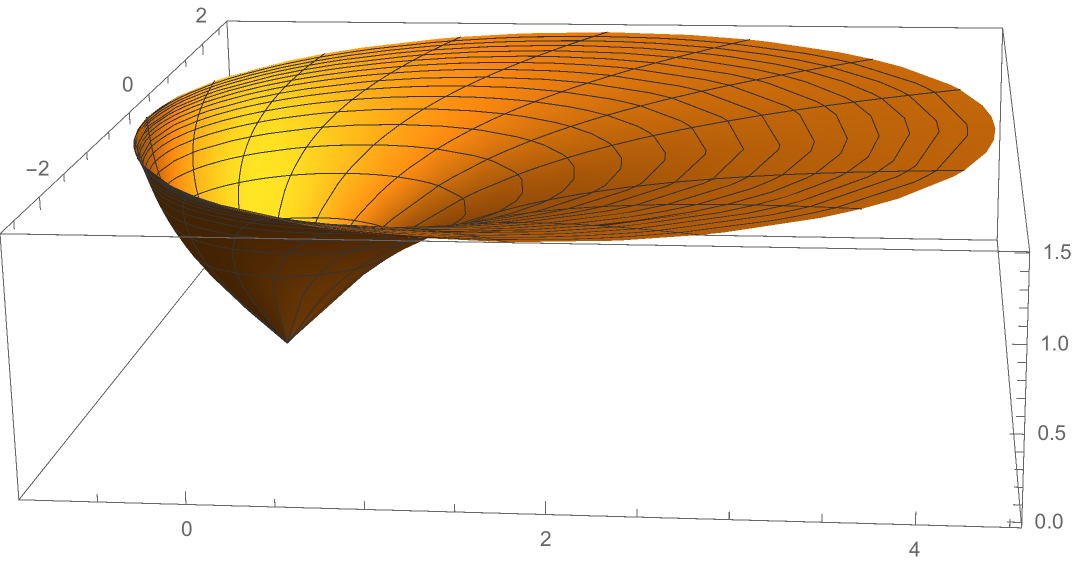}  \includegraphics[width=.7\textwidth]{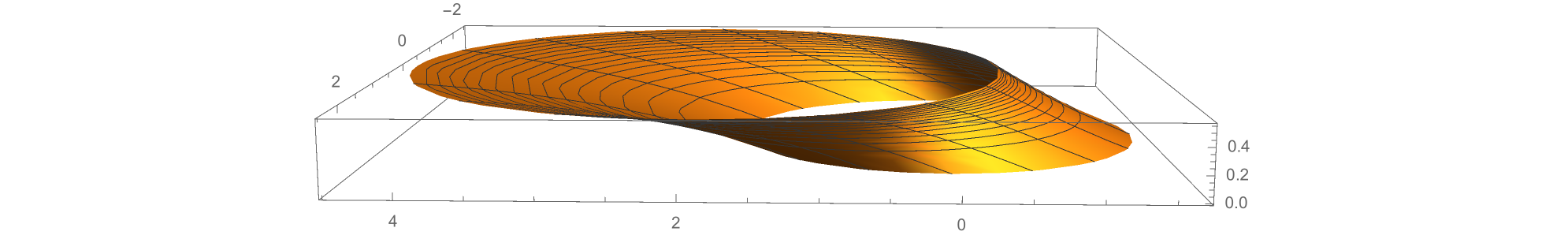}
  \caption{Theorem \ref{pr7}. Left: case $\lambda^2<4$. Right:  case $\lambda<-2$ and $q\in (q_2,\infty)$}\label{fig6}
\end{figure}

Focusing on the Riemann ZMC surfaces of spacelike type, we conclude:

\begin{corollary}\label{c0}
Any   spacelike Riemann ZMC surface foliated by    circles contained in   spacelike planes is included in a horizontal slab $0<x_3\leq z_0$. If $x_3\rightarrow 0$, the surface converges to a conelike point or a circle. At  the height $x_3=z_0$, the surface is a straight-line line $L$ orthogonal to $\Pi$. Furthermore, if $x_3\rightarrow z_0$, the surface is asymptotic to the horizontal plane of equation $x_3=z_0$.
\end{corollary}

We finish this section obtaining properties of symmetries of  the above Riemann ZMC examples.  As a consequence of Proposition \ref{pr33}, we deduce the existence of symmetries about a vertical plane.

\begin{corollary}[symmetry I]\label{c1}
Any Riemann ZMC example foliated by circles contained in   spacelike planes is symmetric about the plane containing the centers of the circles of the foliation.
\end{corollary}

  Suppose that $M$ is a spacelike Riemann ZMC example foliated by circles contained in   spacelike planes.  Up to rotations and dilations,  we can assume that the plane  containing the centers of the circles is the plane $\Pi$ of equation $x_2=0$. By Corollary \ref{c0},  $M$ is asymptotic to the plane $x_3=z_0$ and $M\cap\{x_3=z_0\}$  is a straight-line $L$ orthogonal to $\Pi$. We  reflect  $M$ about  $L$ and we want to apply   the Schwarz's reflection principle in order to extend analytically $M$ along $L$. The   Schwarz's  reflection principle   is due   to the reflection principle of harmonic functions (\cite{os}), which can be easily extended for  maximal surfaces in $\l^3$. We need to assure that $M$ is spacelike around $L$. The straight-line $L$  is obtained letting $q\rightarrow\infty$ with  $W>0$ except if $\cos(v)=1$ and the parameter $\lambda$ satisfies $0\leq \lambda\leq 2$. In such a case, the surface is spacelike around $L$ except at the point where $\cos(v)=1$, which coincides with the intersection point $\Pi\cap L$. Definitively, we have established the following result.

\begin{corollary}[symmetry II] \label{c2}
If $M$ is a spacelike Riemann ZMC example foliated by circles contained in spacelike planes, then $M$ contains a straight-line $L$ orthogonal to the plane $\Pi$ and $M$ can be reflected analytically across $L$.
\end{corollary}

In figure \ref{fig4},  the surface of Theorem \ref{pr3}, case (1), has been extended  by a reflection about the line $L$.

 \begin{figure}[hbtp]
 \begin{center}
 \includegraphics[width=.8\textwidth]{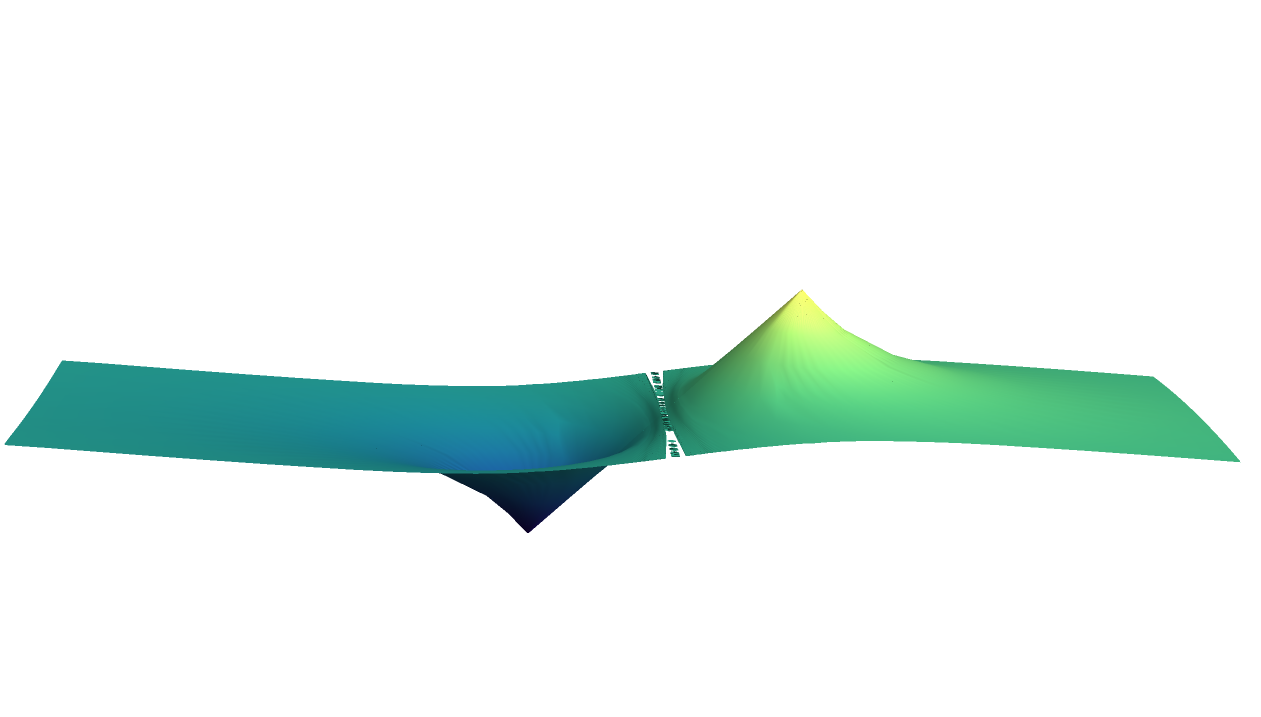}
 \end{center}
 \caption{The surface parametrized by   (\ref{ex1})   has been extended by the Schwarz's reflection principle}\label{fig4}
\end{figure}

Now, we focus in those spacelike Riemann ZMC examples converging to conelike points: Theorem \ref{pr3}, case (1) and   Theorem \ref{pr7}, cases (2) and (4). We know that as $x_3\rightarrow 0$, the surface converges to a conelike point $P$.  Once that we have reflected the surface $M$ about the straight-line $L$, if $\mathcal{R}$ is the reflection across $L$,   the surface  $M^*=M\cup\mathcal{R}(M)$  is included in the slab $0<x_3< 2z_0$. If $x_3\rightarrow 2z_0$, the surface $M^*$ converges to the conelike point $\mathcal{R}(P)$. By means of  the discrete group of translations generated by the vector $\overrightarrow{P\mathcal{R}(P)}$, we  produce copies of $M^*$  obtaining a periodic maximal surface: see figure \ref{fig7}.

\begin{corollary}[symmetry III]\label{c3}
 Let $M$ be a spacelike Riemann ZMC example foliated by pieces of circles contained in   spacelike planes. Suppose that $M$ is bounded by a conelike point and a straight-line $L$ orthogonal to the plane $\Pi$.  Then $M$ can be reflected across $L$ and repeated by translations obtaining  a periodic maximal surface foliated by  pieces of  circles, a discrete set of straight-lines in horizontal planes  and  a discrete set of conelike points. Furthermore, the surface is asymptotic to horizontal planes  at the heights where are situated the straight-lines.
\end{corollary}

\begin{figure}[hbtp]
\begin{center}
\includegraphics[width=.8\textwidth]{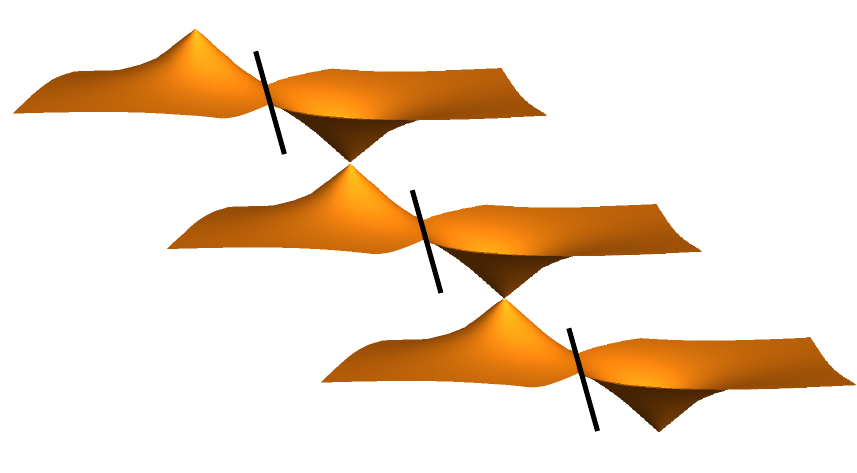}
\end{center}
     \caption{Case $\lambda>2$  in Theorem \ref{pr7}:  the surface has been extended by reflection across the line $L$ and  by translations}\label{fig7}
\end{figure}

Finally, we study the spacelike Riemann ZMC examples that converge to a circle if $x_3\rightarrow 0$: see    Theorem \ref{pr3}, case (2) and Theorem \ref{pr7}, case (3). In such a case, the surface  can be extended by reflection across $L$ by the Schwarz's reflection principle, obtaining  a spacelike Riemann ZMC example included in the slab  $0<x_3< 2z_0$ and  converging   to two circles if  $x_3\rightarrow 0$ and if $x_3\rightarrow 2z_0$.   See figure \ref{fig7b}.

\begin{corollary}\label{c4}
Let $M$ be a spacelike Riemann ZMC example foliated by  pieces of  circles contained in   spacelike planes. Suppose that $M$   converges to a circle if $x_3\rightarrow 0$ and contains a   straight-line $L$ orthogonal to the plane $\Pi$ at $x_3=z_0$.  Then $M$ can be reflected across $L$  obtaining a spacelike  surface contained in the slab 
$0<x_3< 2z_0$ and   foliated by  pieces of  circles. Furthermore,  the surface contains a straight-line and converges to two circles as $x_3\rightarrow 0$ and as $x_3\rightarrow 2z_0$. The surface is asymptotic to the horizontal plane that contains $L$.
 \end{corollary} 

 \begin{figure}[hbtp]
 \begin{center}
\includegraphics[width=.9\textwidth]{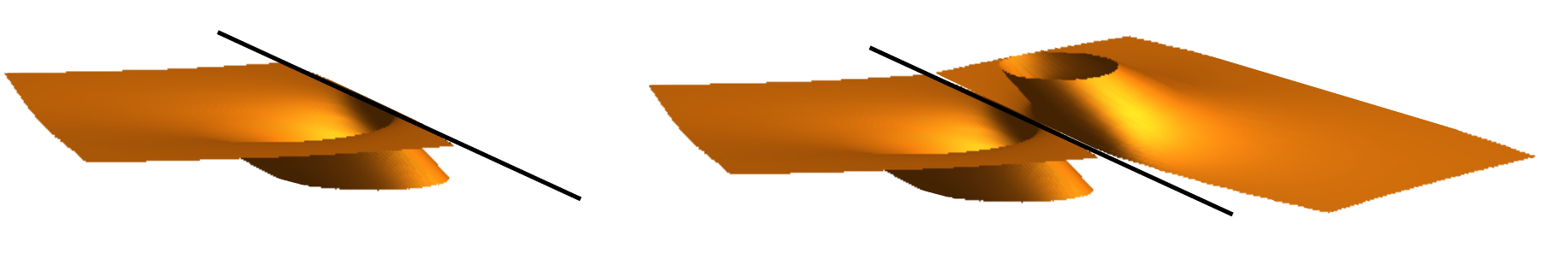} 
\end{center}
  \caption{A spacelike Riemann ZMC example of Theorem \ref{pr3}, case (2). Left: the surface is bounded by   a straight-line $L$ and a circle. Right: the same surface after a reflection across $L$}\label{fig7b}
\end{figure}

\section{ Riemann ZMC examples foliated by   spacelike circles contained in timelike planes}\label{sec4}

In this section, we study   ZMC surfaces in $\l^3$    foliated by spacelike circles  contained     in parallel timelike planes. The arguments and computations follow the same steps than in the previous section.   In order to be not repeated, we will omit the details. 

Without loss of generality, we can  assume   that the planes of the foliation are  parallel to the plane of equation $x_1=0$. Let $M$ be a such surface and consider the height $x_1$ of the plane as a parameter of the foliation. Let $x=x_1$ and let $(x,\alpha(x))=(x,\alpha_2(x),\alpha_3(x))$ be the center of the circle $M\cap\{x_1=x\}$.   A Lorentzian circle in the $x_2 x_3$-plane is a hyperbola that can be parametrized as $\alpha(s)=p+r(0,\sinh(s),\cosh(s))$, $s\in\r$, or $\beta(s)=p+r(0,\cosh(s),\sinh(s))$, $s\in\r$, where $p\in\l^3$ and $r>0$. Since in this section we are considering spacelike circles, then right  choice  is  the hyperbola $\alpha$. We notice now  that      each circle of the foliation is not compact.  

The surface $M$ can be expressed as   $M\subset F^{-1}(\{0\})$, where $F:(a,b)\times\r^2\rightarrow\r$ is  
$$F(x,y)=(x_2-\alpha_2(x)^2-(x_3-\alpha_3(x))^2+r(x)^2.$$
 Let  $y=(x_2,x_3)$. We identify the factor $\r^2$ of the domain of $F$ as $\{0\}\times\r^2$ endowed with the induced metric  $\langle,\rangle=dx_2^2-dx_3^2$.  The  equation (\ref{elm1}) is 
$$
2r^{2}+r^{2}F_{xx}-(r^{2})'F_{x}=0.
$$
Again, we divide this equation  by $r^{4}$, obtaining 
$$
2\int^{x}  \frac{1}{r(u)^{2}}du+\frac{F_{x}}{{r(x)^2}}=c(y)
$$
for a  function $c=c(y)$  depending only the variable $y$. We deduce that  there is a vector $\vec{a}=(a_2,a_3)$    such that $\alpha'(x)={r(x)^2} \vec{a}$. 
Integrating with respect to  $x$,  $$\alpha(x)=m(x)\vec{a}, \quad m(x)=\int^{x} r(u)^2du.$$
If $\vec{a}=(0,0)$, the centers of the circles are included in the $x_1$-axis,  $M$ is a surface of revolution and the $x_1$-line is the rotational axis. This case was discarded from the beginning. By taking  $q=r^2$, we obtain an ordinary differential equation on $q$, namely, 
\begin{equation}\label{ss}
2\langle \vec{a},\vec{a}\rangle q^3-q'^2+q(2+q'')=0.
\end{equation}

As in the previous section, we study the case that the radii   coincide in all circles of the foliation.  

\begin{proposition}[Case of constant radii]\label{pr11}
The only Riemann ZMC examples   foliated by spacelike circles with constant radii contained in   timelike planes   are parametrized by 
\begin{equation}\label{tc2}
X(x,v)=r^2x(0,\vec{a})+(x,r\sinh(v), r\cosh(v)),
\end{equation}
where $x,v\in\r$ and $\vec{a}\in\r^2\setminus\{0\}$. The surfaces are timelike except at the points  satisfying $\langle\vec{a},(\cosh(v),\sinh(v))\rangle=0$,  which form a lightlike straight-line.
\end{proposition}

\begin{proof}
 Let us observe that the vector $\vec{a}$  may have any causal character. Now (\ref{ss}) is $2\langle \vec{a},\vec{a}\rangle q^3+ 2q=0$, so  there are not solutions if $\langle\vec{a},\vec{a}\rangle\geq 0$. If $\langle\vec{a},\vec{a}\rangle<0$, then $q^2=-1/\langle \vec{a},\vec{a}\rangle $ and the  parametrization of the surface is   (\ref{tc2}): see figure \ref{fig1}, right. 
 The first derivatives of $X$ are $X_{x}=r^2(0,\vec{a})+(1,0,0)$ and $X_v=r(0,\cosh(v),\sinh(v))$, obtaining  $g_{11}=0$  and $g_{12}=r^3\langle\vec{a},(\cosh(v),\sinh(v))\rangle$. Thus the surface is timelike, except at the points satisfying $\langle\vec{a},(\cosh(v),\sinh(v))\rangle=0$.  These points  form a lightlike straight-line. Indeed,   this curve parametrizes as $x\mapsto X(x,v)$ so $X'(x)=r^2(0,\vec{a})+(1,0,0)$ and $X'(x)$ is lightlike because $1+r^4\langle \vec{a},\vec{a}\rangle=0$.
 \end{proof}

From now, we suppose that the radii of the foliation circles are not constant.  

\begin{theorem} \label{pr77}
 Riemann ZMC examples foliated by spacelike circles contained in   timelike planes form a one-parameter family of surfaces depending on a parameter $\lambda\in\r$ and parametrize as 
$$X(q,v)=(x(q),m(q)\vec{a})+\sqrt{q}(0,\sinh(v),\cosh(v)),$$
where
\begin{equation}\label{mzh}
x(q)=\frac{1}{2}\int ^q\frac{du}{\sqrt{- \langle\vec{a},\vec{a}\rangle u^{3}+\lambda u^{2}+u}},\quad m(q)=\frac{1}{2}\int^q\frac{u\ du}{\sqrt{-\langle \vec{a},\vec{a}\rangle u^3+\lambda u^2+u}}.                                                  
\end{equation}
The curve formed by the centers of the circles   is contained in a plane and the surface is symmetric about this plane.  
\end{theorem} 

\begin{proof}
Using (\ref{ss}),
$$\left(\frac{(q')^{2}}{q^{2}}\right)'= -4q'\left( \langle\vec{a},\vec{a}\rangle+\frac{1}{q^{2}}\right)$$
and integrating with respect to $x$, 
\begin{equation}\label{difh}
\frac{(q')^{2}}{q^{2}}=4\left(- \langle\vec{a},\vec{a}\rangle q+\frac{1}{q}\right)+4\lambda,
\end{equation}
for a constant $\lambda\in\r$. In particular, the right-hand side of \eqref{difh} must be non-negative. Now we have
$$q'=\frac{dq}{dx}=2\sqrt{- \langle\vec{a},\vec{a}\rangle q^{3}+\lambda q^{2}+q}$$
and
$$\frac{dx}{dq}=\frac{1}{2}\frac{1}{\sqrt{- \langle\vec{a},\vec{a}\rangle q^{3}+\lambda q^{2}+q}}.$$
As our new parameter is $q$,   the center of the circle $M\cap\{x_1=x\}$ is
$(x,\alpha(x))=(x,m(x)\vec{a})$. This proves that this curve is contained in the plane determined  by the $x_1$-axis and the vector $\vec{a}$ (recall f $\vec{a}\not=(0,0)$). 
\end{proof}

We identify   the vector $\vec{a}=(a_2,a_3)$  with $(0,\vec{a})\in\l^3$ contained in the $x_2x_3$-plane. After a rotation about the $x_1$-axis and a dilation, we may suppose that the vector $\vec{a}$ is $(1,0)$, $(0,1)$ or $(1,1)$. We discuss the three cases.

 \subsection{Case $\vec{a}=(1,0)$}
 The parametrization of the surface is 
   $$X(q,v)=(x(q),m(q),0)+\sqrt{q}(0,\sinh(v),\cosh(v)),$$
  where 
\begin{equation}\label{mz10}
 x(q)=\frac{1}{2}\int ^q_{q_0}\frac{du}{\sqrt{-u^{3}+\lambda u^{2}+u}},\quad m(q)=\frac{1}{2}\int^q_{q_0}\frac{u\ du}{\sqrt{-u^3+\lambda u^2+u}}.
\end{equation}
The sign of $W$ is determined by the expression  
 \begin{equation}\label{hspa}
q-q\sinh^2 (v)+2\sqrt{-q^2+\lambda q+1}\sinh(v)-\lambda.
\end{equation}
We analyze  when    $-u^2+\lambda u+1$ is positive in order to determine the lower limit $q_0$ in the integrals (\ref{mz10}). The roots of the function   $-u^3+\lambda u^2+u=0$ are $0$  and 
$$q_1=\frac{\lambda-\sqrt{\lambda^2+4}}{2},\quad q_2=\frac{\lambda+\sqrt{\lambda^2+4}}{2},$$
with $q_1<0<q_2$. Then  $q\in(0,q_2)$ and the surface  contains regions with the three causal character according to (\ref{hspa}). It is immediate that the   integral $x(q)$ in (\ref{mz10}) for $q_0=0$ is finite. Let $c=x_1(q_2)$.  

\begin{proposition} If $\vec{a}=(1,0)$, then  the surface lies contained in the  vertical slab  $ 0<x_1<c$  and the surface  converges to one one point if $x_1\rightarrow 0$.  
\end{proposition}

 \subsection{Case $\vec{a}=(0,1)$}
 The parametrization of the surface is 
  $$X(q,v)=(x(q),0,m(q))+\sqrt{q}(0,\sinh(v),\cosh(v)),$$
  where 
\begin{equation}\label{mz01}
x(q)=\frac{1}{2}\int ^q_{q_0}\frac{du}{\sqrt{u^{3}+\lambda u^{2}+u}}, \quad m(q)=\frac{1}{2}\int^q_{q_0}\frac{u\ du}{\sqrt{u^3+\lambda u^2+u}}. 
\end{equation}
The roots of $u^3+\lambda u^2+u$ are $0$  and 
$$q_1=\frac{-\lambda-\sqrt{\lambda^2-4}}{2},\quad q_2=\frac{-\lambda+\sqrt{\lambda^2-4}}{2}.$$
The spacelike condition  $W>0$ is equivalent to
\begin{equation}\label{hspa2}
(1+\cosh^2 (v))q+2 \cosh(v)\sqrt{q^2+\lambda q+1}+\lambda <0.
\end{equation}
Therefore   $\lambda$ is negative. As in the previous section, the integrals  (\ref{mz01}) can be explicitly integrated if $\lambda=\pm 2$.

\begin{proposition} \label{pr9}
If $\vec{a}=(0,1)$, the special cases $\lambda=\pm2$ provide   parametrizations of Riemann ZMC examples in terms of elementary functions, namely, 
\begin{enumerate}
\item Case $\lambda=2$. The surface is   
$$X(r,v)=(\arctan(r),0,r-\arctan(r))+r(0,\sinh(v),\cosh(v)),$$
where $r>0$, $v\in\r$. The surface  is timelike converging to a point if $r\rightarrow 0$.  See figure \ref{fig9}, left.

\item Case $\lambda=-2$. The surface   is  
$$X(r,v)=\left(\log\left|\frac{1-r}{r+1}\right|,0,r+\frac12\log\left|\frac{1-r}{r+1}\right|\right)+r(0,\sinh(v),\cosh(v)),$$
where $r>0$, $r\not=1$ and $v\in\r$. If $0<r<1$ the surface contains regions of spacelike and timelike points, but if $r>1$, the surface is timelike. See figure \ref{fig9}, right. 

\end{enumerate}
\end{proposition}

 \begin{figure}[hbtp]
 \begin{center}
 \includegraphics[width=.3\textwidth]{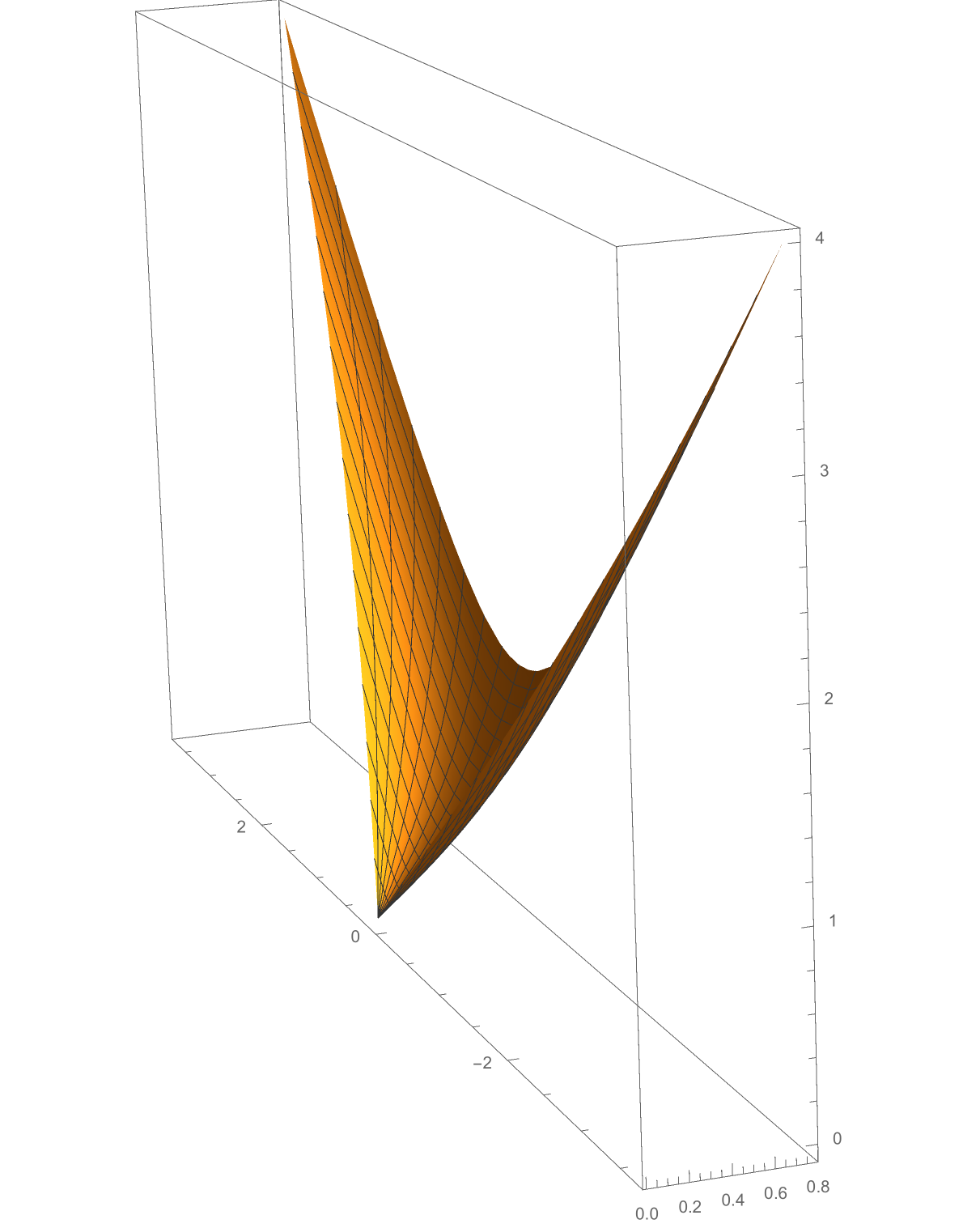}\quad \includegraphics[width=.5\textwidth]{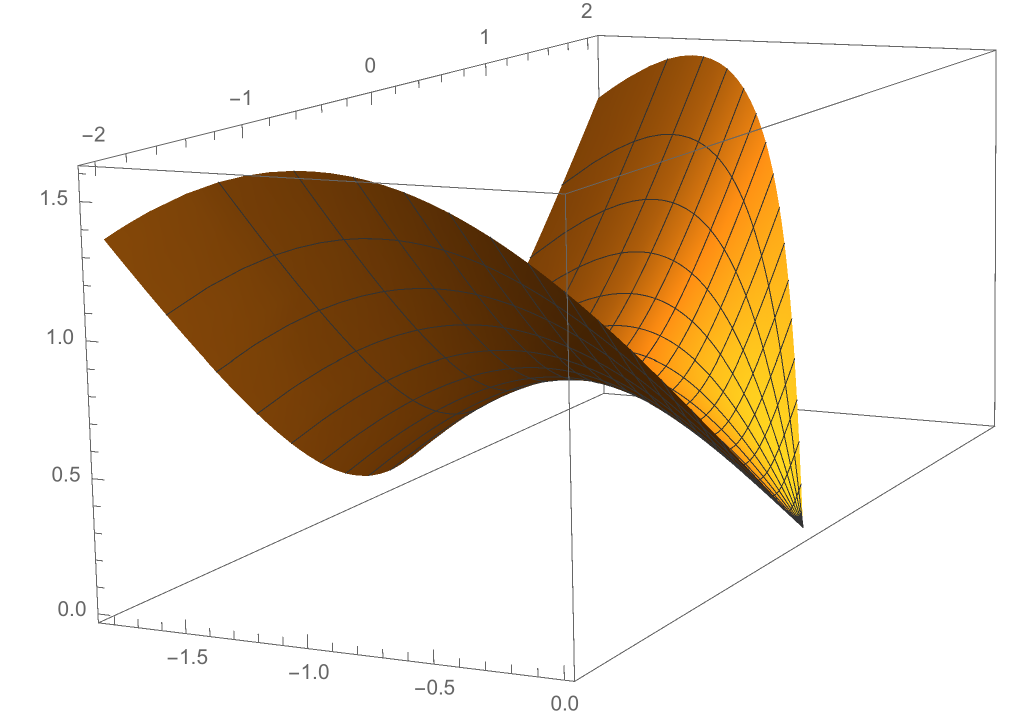}
 \end{center}
\caption{Surfaces of Proposition \ref{pr9}: Riemann ZMC surfaces in $\l^3$ foliated by spacelike hyperbolas in   timelike planes}\label{fig9}
\end{figure}

If we now consider the integrals (\ref{mz01}), the discussion is similar as in Section \ref{sec3}. We need to distinguish on the parameter $\lambda$ according if $\lambda>2$,  $\lambda^2-4<0$ and $\lambda<-2$.

\begin{proposition} If $\vec{a}=(0,1)$ and $\lambda^2\not= 4$, then the surface is included in a  vertical slab $0<x_1<c$. If $\lambda<-2$ and $q\in (q_2,\infty)$, then the surface converges to a hyperbola if  $q\rightarrow q_2$. In the rest of cases,   the surface converges to a point  as $x_1\rightarrow 0$.  \end{proposition}

\begin{proof}
\begin{enumerate}
\item Case $ \lambda>2$. Then $q_1<q_2<0$ and we  take $q_0=0$ to be the lower limit in (\ref{mz01}). From (\ref{hspa2}) we deduce that the surface is   timelike. It is also immediate that at $q=0$, the surface converges to a point and if $q\rightarrow  \infty$, the   integral for $x(q)$ is finite, that is, 
$\lim_{q\rightarrow\infty} x(q)<\infty$.
\item Case $\lambda^2<4$. Since the polynomial $u^2+\lambda u+1$ has not real roots,    $u^3+\lambda u^2+u$ is positive for $u>0$. We infer that  we  can take $q_0=0$ to be the lower limit in  (\ref{mz01}). Again, we have 
$\lim_{q\rightarrow\infty} x(q)<\infty$.
\item Case $\lambda<-2$. Then $0<q_1<q_2$, hence that $u^3+\lambda u^2+u>0$   in  $(0,q_1)\cup(q_2,\infty)$. If $q\in (0,q_1)$, let $q_0=0$ and if $q\in (q_2,\infty)$, the lower integration limit is $q_0=q_2$. Moreover, the improper integral $x(\infty)$ is finite.
\end{enumerate}
 \end{proof} 

 \subsection{Case $\vec{a}=(1,1)$}
  In this case  $\langle \vec{a},\vec{a}\rangle=0$. The parametrization of the surface is 
  \begin{equation}\label{a11}
  X(r,v)=(x(r),m(r),m(r))+r(0,\sinh(v),\cosh(v)),
  \end{equation}
where if $q=r^2$, then 
\begin{equation}\label{mz11}
 x(q)=\frac{1}{2}\int ^q_{q_0}\frac{du}{\sqrt{\lambda u^{2}+u}},\quad m(q)=\frac{1}{2}\int^q_{q_0}\frac{u\ du}{\sqrt{\lambda u^2+u}}
\end{equation}
and $\lambda q+1>0$. The computation of $W$ yields
\begin{equation}\label{w11}
W=-\frac{q e^{-2v}}{4(1+\lambda q)}(\lambda e^{2v}+2e^v\sqrt{1+\lambda q}+q).
\end{equation} 
In particular, if the surface is spacelike if $\lambda<0$, and the surface if timelike if   $\lambda>0$. This case $\vec{a}=(1,1)$ provides explicit parametrizations   by integrating    (\ref{mz11}).

  \begin{proposition} \label{pr1111}
   Riemann ZMC examples foliated by spacelike circles contained in   timelike planes corresponding to the choice $\vec{a}=(1,1)$ have the following explicit parametrizations \eqref{a11}, where the functions $x(r)$ and $m(r)$ are the following.
  \begin{enumerate}
  \item Case $\lambda=0$.  Here   $x(r)=r$ and $m(r)= r^3/3$. The surface is     timelike and converges to a point as $r\rightarrow 0$.
  \item Case $\lambda>0$.  The surface     converges to a point as $r\rightarrow 0$.  Here
$$x(r)= \frac{1}{\sqrt{\lambda}}\operatorname{arsinh}(\sqrt{\lambda }r),\quad 
 m(r) =\frac{1}{2\lambda^{3/2}}\left(r\sqrt{\lambda r^2+1 }-\operatorname{arsinh}(\sqrt{\lambda }r)\right). 
$$
  \item Case $\lambda<0$. Here  
 $$x(r)=\frac{1}{\sqrt{-\lambda}}\arcsin(\sqrt{-\lambda}r),\quad 
 m(r)=\frac{1}{2\lambda}r\sqrt{\lambda r^2+1}+\frac{1}{4(-\lambda)^{3/2}}\arcsin(-2\lambda r^2-1).$$
   Depending on the values $v$ in (\ref{w11}), the surface   contains spacelike, lightlike and timelike regions.    The surface is included in the slab $0<x_1<\pi/(2\sqrt{-\lambda})$ and converges  to a point as $r\rightarrow 0$.   
  \end{enumerate}
  
     \end{proposition}
  

\section{ Riemann ZMC examples foliated by timelike circles contained in timelike planes}\label{sec5}

  In this section, we  study  ZMC surfaces in $\l^3$  foliated by timelike circles contained in parallel timelike planes.  Recall that any curve contained in a spacelike surface is a spacelike curve. Thus all surfaces of this section will be timelike and, perhaps, can be extended to lightlike regions. Again, we will assume that the planes of the foliation are parallel to the plane of equation $x_1=0$. We know that a  timelike circle   in the plane $x_1=0$ is parametrized by $\beta(s)=p+r(0,\cosh(s),\sinh(s))$.  Let $M$ be a such surface and  take  the $x_1$-coordinate as the parameter of the circles of the  foliation. Then   $M$ is included in $F^{-1}(\{0\})$, where $F:(a,b)\times\r^2\rightarrow\r$ is
$$F(x,y)=-(x_{2}-\alpha_{2}(x))^{2}+(x_3-\alpha_{3}(x))^{2}+{r(x)^2}. $$
Again, let $y=(x_2,x_3)$,  $\alpha=(\alpha_{2},\alpha_{3})$ and $\langle,\rangle=dx_2^2-dx_3^2$ the induced metric in the $x_2x_3$-plane. The zero mean curvature equation (\ref{elm1}) is  
$$F_{x}^{2}+r^{2}(2-F_{xx})-2F_{x}\langle y-\alpha,\alpha'\rangle=0.$$
If we repeat the same steps of the previous section,     there is a vector $\vec{a}=(a_2,a_3)$, which we identify in the $x_2x_3$-plane  as $\vec{a}=(0,a_2,a_3))\in\l^3$,  such that
$$\alpha(x)=m(x)\vec{a},\quad  m(x)=\int^{x} r(u)^2du$$
and  
\begin{equation}\label{551}
2 \langle\vec{a},\vec{a}\rangle r^{6}+(r^{2})'^{2}+r^{2}(2-(r^{2})'')=0.
\end{equation}
When the radii of the circles are constant, the above equation is $ \langle\vec{a},\vec{a}\rangle r^{4}+ 1=0$. In particular, $\langle\vec{a},\vec{a}\rangle<0$, so the vector $\vec{a}$ is timelike.

\begin{proposition}[Case of constant radii]\label{pr111}
The only Riemann ZMC examples   foliated by timelike circles with constant radii contained in   timelike planes   are parametrized by 
$$X(x,v)=x(1,r^2\vec{a}) +r(0,\cosh(v),\sinh(v)),$$
where $x,v\in\r$ and $\vec{a}\in\r^2\setminus\{0\}$. The surfaces are timelike except at the points $v=0$, which form a straight-line of lightlike points. See figure \ref{fig1}, right.
\end{proposition}

We assume that the radii of the foliation circles are  not constant. 
 
 \begin{theorem}\label{pr777}
   Riemann ZCM examples foliated by timelike circles form a one-parameter family of surfaces depending on a parameter $\lambda\in\r$ and parametrize as  
 \begin{equation}\label{hparamt}
X(q,v)=(x(q),m(q)\vec{a})+\sqrt{q}(0,\cosh(v),\sinh(v)),
\end{equation}
 where $\vec{a}$ is any timelike vector in the $x_2x_3$-plane and 
 \begin{equation}\label{mz3}x(q)=\frac{1}{2}\int^{q} \frac{1}{\sqrt{ \langle\vec{a},\vec{a}\rangle u^{3}+\lambda u^{2}-u}} du,\quad m(q)=\frac{1}{2}\int^{q} \frac{u}{\sqrt{ \langle\vec{a},\vec{a}\rangle u^{3}+\lambda u^{2}-u}} du.
\end{equation}
The curve formed by the centers of the foliation   is contained in a plane and   the surface is symmetric about this plane.  
 \end{theorem}
 
 \begin{proof}
 
 Let $q=r^2$. Then (\ref{551})   is
$2 \langle\vec{a},\vec{a}\rangle q^{3}+(q')^{2}+q(2-q'')=0$. Using this equation, we obtain 
$$\left(\frac{(q')^{2}}{q^{2}}\right)'=4q'\left( \langle\vec{a},\vec{a}\rangle+\frac{1}{q^{2}}\right),$$
and integrating with respect to $x$, 
$$\frac{(q')^{2}}{q^{2}}=4\left( \langle\vec{a},\vec{a}\rangle q-\frac{1}{q}\right)+4\lambda,\quad \lambda\in\mathbb{R}.$$
In particular, the right-hand side of this equation  must be non-negative. Now we have
$$
q'=\frac{dq}{dx}=2\sqrt{ \langle\vec{a},\vec{a}\rangle q^{3}-q+\lambda q^{2}}
$$
and 
$$
\frac{dx}{dq}=\frac{1}{2}\frac{1}{\sqrt{ \langle\vec{a},\vec{a}\rangle q^{3}+\lambda q^{2}-q}}.
$$
Using $q$ as a  new parameter is $q$, it follows (\ref{mz3}), hence (\ref{hparamt}). 
 \end{proof}

Again   the arguments  are similar with Section \ref{sec4}. The family of   timelike  Riemann ZMC examples foliated by timelike circles contained in   timelike planes  depends on a real parameter $\lambda$.  After a rigid motion and a dilation, we have three cases according the value of the vector $\vec{a}$. 

 \subsection{Case $\vec{a}=(1,0)$}
 The parametrization of the surface is 
 \begin{equation}\label{52a10}
 X(q,v)=(x(q),m(q),0)+\sqrt{q}(0,\cosh(v),\sinh(v)),
 \end{equation}
  where 
\begin{equation}\label{5mz10}
 x(q)=\frac{1}{2}\int ^q_{q_0}\frac{du}{\sqrt{u^{3}+\lambda u^{2}-u}},\quad m(q)=\frac{1}{2}\int^q_{q_0}\frac{u\ du}{\sqrt{u^3+\lambda u^2-u}}.
\end{equation}
Then 
 \begin{equation}\label{5hspa}
W=-\frac{q\left( (1+\cosh^2 (v))q-2\sqrt{q^2+\lambda q-1}\cosh(v)+\lambda\right)}{4(q^2+\lambda q-1)}.
\end{equation}
 The roots of   $u^3+\lambda u^2-u=0$ are $0$  and 
$$q_1=\frac{-\lambda-\sqrt{\lambda^2+4}}{2},\quad q_2=\frac{-\lambda+\sqrt{\lambda^2+4}}{2}.$$

Since $q_1<0<q_2$, then $q\in(q_2,\infty)$. Hence we can take the lower limit in   (\ref{5mz10}) as $q_2$. From (\ref{5hspa}), we have $W<0$ and the surface is timelike. On the other hand, the improper integral $x(q)$ in (\ref{5mz10}) with $q=\infty$ is convergent, that is,   $\lim_{q\rightarrow\infty}x(q)=c<\infty$, and the same occurs if $q\rightarrow q_2$. Accordingly, in the limit $x_1(q_2)$, the surface is a hyperbola.

\begin{proposition} If $\vec{a}=(1,0)$,   the surface is timelike and  parametrized as  
 (\ref{52a10}). Here  the lower integration limit is $q_2$ and the functions $x(q)$ and $m(q)$ are determined in (\ref{5mz10}). The surface is contained in  a  vertical slab $0\leq x_1<  c$ and  the surface converges to a hyperbola as  $x_1\rightarrow 0$.
 \end{proposition}

 \subsection{Case $\vec{a}=(0,1)$} 
 
 The parametrization of the surface is 
\begin{equation}\label{52a01}
X(q,v)=(x(q),0,m(q))+\sqrt{q}(0,\cosh(v),\sinh(v)),
\end{equation}
  where 
\begin{equation}\label{5mz01}
x(q)=\frac{1}{2}\int ^q_{q_0}\frac{du}{\sqrt{-u^{3}+\lambda u^{2}-u}}, \quad m(q)=\frac{1}{2}\int^q_{q_0}\frac{u\ du}{\sqrt{-u^3+\lambda u^2-u}}.
\end{equation}
The roots of the equation $-u^3+\lambda u^2-u$ are $0$  and 
$$q_1=\frac{\lambda-\sqrt{\lambda^2-4}}{2},\quad q_2=\frac{\lambda+\sqrt{\lambda^2-4}}{2}.$$
Then
\begin{equation}\label{5hspa2}
W=-\frac{q\left( (-1+\sinh^2 (v))q-2\sqrt{-q^2+\lambda q-1}\sinh(v)+\lambda\right)}{4(-q^2+\lambda q-1)}.\end{equation}

Here   the   case  $\lambda^2=4$ must be discarded because $-u^3+\lambda u^2-u\leq 0$.  If $\lambda^2\not=4$,   we  distinguish the cases  $\lambda>2$,  $\lambda^2-4<0$ and $\lambda<-2$.  However, for  the cases $\lambda^2-4<0$ and $\lambda<-2$, the polynomial $-u^3+\lambda u^2-u$ is negative when $u>0$. Consequently, the only possibility is that   $ \lambda>2$. In such a case, $0<q_1<q_2$ and $-u^3+\lambda u^2-u>0$ if $q\in (q_1,q_2)$. We take $q_1$ to be the lower limit in (\ref{5mz01}). From (\ref{5hspa2}),  we deduce that the surface is   timelike.

\begin{proposition} If $\vec{a}=(0,1)$,   then  $\lambda>2$. The surface is parametrized as (\ref{52a01}), where the lower integration limit is $q_1$ and the functions $x(q)$ and $m(q)$ are given by (\ref{5mz01}).  The surface is contained  in a vertical slab  $0< x_1< c$, with $c=x(q_2)$, and in the limits $x_1=0$ and $x_1=c$, the surface is formed by two hyperbolas. 
\end{proposition}

 \subsection{Case $\vec{a}=(1,1)$}
The integrals in (\ref{mz3}) are now
 $$ x(q)=\frac{1}{2}\int ^q_{q_0}\frac{du}{\sqrt{\lambda u^{2}-u}},\quad m(q)=\frac{1}{2}\int^q_{q_0}\frac{u\ du}{\sqrt{\lambda u^2-u}}.
$$
 In particular, $\lambda$ must be a positive number. From a direct integration, we prove the following proposition.
  
  \begin{proposition} \label{pr21}
   Riemann ZMC examples foliated by timelike circles contained in   timelike planes corresponding to the case $\vec{a}=(1,1)$  parametrize as  
    $$X(r,v)=(x(r),m(r),m(r))+r(0,\cosh(v),\sinh(v)),$$
    where $r>1/\sqrt{\lambda}$, $v\in\r$,    $\lambda>0$ and 
$$x(r)= \frac{1}{\sqrt{\lambda}}\operatorname{arcosh} (\sqrt{\lambda }r),\quad 
  m(r) =\frac{\sqrt{\lambda } r\sqrt{\lambda  r^2-1}+\operatorname{arsinh} \left(\sqrt{\lambda  r^2-1}\right)}{2 \lambda ^{3/2}}. 
$$
The surface is contained in the halfspace $x_1>0$ and   converges to a hyperbola as   $x_1\rightarrow 0$.
     \end{proposition}


\end{document}